\documentclass[12pt]{amsart}
\usepackage{amsmath,amsthm,amsfonts,amssymb}

\newcommand{\C}{{\mathbb C}}

\newcommand{\Z}{{\mathbb Z}}
\renewcommand{\P}{{\mathbb P}}

\newcommand{\R}{{\mathbb R}}

\newtheorem{definition}[equation]{Definition}
\newtheorem{corollary}[equation]{Corollary}
\newtheorem{conjecture}[equation]{Conjecture}
\newtheorem{example}[equation]{Example}
\newtheorem{remark}[equation]{Remark}
\newtheorem{lemma}[equation]{Lemma}
\newtheorem{proposition}[equation]{Proposition}
\newtheorem{theorem}[equation]{Theorem}
\newtheorem{question}[equation]{Question}
\newtheorem*{maintheorem}{Main Theorem}
\numberwithin{equation}{section}
\begin{document} 
\title[On $h$-principle and specialness]{%
On $h$-principle and specialness for complex projective manifolds.
}
\author {Frederic Campana \& J\"org Winkelmann}
\begin{abstract} We show that a complex projective manifold $X$ which satisfies the Gromov's $h$-principle is `special' in the sense of \cite{C01} (see definition 2.1 below), and raise some questions about the reverse implication, the extension to the quasi-K\" ahler case, and the relationships of these properties to the `Oka' property. The guiding principle is that the existence of many Stein manifolds which have degenerate Kobayashi pseudometric gives strong obstructions to the complex hyperbolicity of $X's$ satisfying the $h$-principle.

\end{abstract}
%
\maketitle
\section{Introduction}
\begin{definition}\label{dip} (M. Gromov) A complex space $X$ is said to {\em satisfy the $h$-principle} (a property abbreviated by: `hP(X)') if:
for every Stein manifold $S$ and every continuous map
$f:S\to X$ there exists a holomorphic map $F:S\to X$
which is homotopic to  $f$.
\end{definition}

The origin of this notion lies in the works of Grauert and Oka. Grauert indeed showed that any holomorphic principal bundle with fibre a complex Lie Group $G$ over a Stein Manifold $S$ has, for any given continuous section $s$, a holomorphic section homotopic to $s$. The classifications of continuous complex and holomorphic vector bundles on $S$ thus coincide. This was established by Oka for complex line bundles. Considering products $G\times S$, Grauert's result also shows that complex Lie groups satisfy the $h$-principle. This has been extended by M. Gromov to `elliptic' (and later to Forstneric `subelliptic') manifolds.
These include homogeneous complex manifolds 
(for example $\P_n$, Grassmannians and tori)
and complements $\C^n\setminus A$
where $A$ is an algebraic subvariety of codimension at least two.
(Sub-)elliptic manifolds contain as many `entire' curves as possible, and
therefore opposite to Brody hyperbolic complex manifolds.
Since `generic' hyperbolicity is conjectured (and sometimes known) to coincide with``general type''
in algebraic geometry, it is thus natural to assume that 
for projective varieties ``fulfilling
the $h$-principle'' is related to being ``special'' as introduced in
\cite{C01}, since `specialness' is conjectured there to be equivalent to $\Bbb C$-connectedness. In this article we investigate these relationships
with particular emphasis on projective manifolds.

The Main result is:

\begin{maintheorem} 
Let $X$ be a complex projective manifold fulfilling the $h$-principle. Then:
\begin{enumerate}
\item
$X$ is {\em special}.
\item
Every holomorphic map from $X$ to a Brody hyperbolic K\"ahler manifold
is constant.
\end{enumerate}
\end{maintheorem}

For an arbitrary complex manifold we prove the statements below.

\begin{theorem}
Let $X$ be a complex  manifold fulfilling the $h$-principle.

\begin{enumerate}
\item
Then is {\em weakly $\C$-connected} (see definition~\ref{dcc}
below).
\item
For every holomorphic map from $X$ to a complex
semi-torus $T$, the Zariski closure of $f(X)$ in $T$
is the translate of a complex sub semi-torus of $T$.
\item
If $X$ is an algebraic variety, its Quasi-Albanese map is
dominant.
\end{enumerate}

\end{theorem}

Let us now recall resp.~introduce some notation.

\begin{definition}\label{dcc} We say that a complex space $X$ is:

\begin{enumerate}
\item
 $\Bbb C$-connected if any two points of $X$ can be connected by a chain of `entire curves' (i.e.,~holomorphic maps from $\C$ to $X$) . 
This property is preserved by passing to unramified coverings and 
images by holomorphic maps. If $X$ is smooth
this property is also preserved under proper modifications.
\item
 `Brody-hyperbolic' if any holomorphic map $h:\Bbb C\to X$ is constant.
\item
 $X$ is said to be `weakly $\Bbb C$-connected' if every holomorphic map $f:X'\to Y$ from any unramified covering $X'$ of $X$ to a Brody-hyperbolic complex space $Y$ induces maps 
$\pi_k(f):\pi_k(X')\to \pi_k(Y)$ between the
respective homotopy groups which are zero for any $k>0$. 

Observe that any holomorphic map $f:X\to Y$ between complex spaces
is constant if $X$ is $\Bbb C$-connected and $Y$ Brody-hyperbolic. Thus $\Bbb C$-connectedness implies `weak $\Bbb C$-connectedness'. Also, any contractible $X$ is `weakly $\Bbb C$-connected'. 

There exists projective smooth threefolds which are `weakly $\Bbb C$-connected', but not $\Bbb C$-connected. An example can be found in \cite{CW}.
\end{enumerate}
\end{definition}

It is easy to verify that  every `subelliptic' manifold 
$X$ is $\Bbb C$-connected. 
Conversely all known examples of  
connected complex manifolds  satisfying the $h$-principle 
admit a holomorphic homotopy equivalence to a `subelliptic'
complex space.

This suggest the following question:

\begin{question}\label{qhpcc} 
Let $X$ be a complex connected manifold.

If $X$ satisfies the $h$-principle, does this imply that there
exists a holomorphic homotopy equivalence
between $X$ and a 
$\Bbb C$-connected complex space $Z$ ?
\end{question}

Since a compact manifold can not be homotopic to a proper analytic subset
for compact manifolds this question may be reformulated as follows:

\begin{question}\label{qhpccK} 
Let $X$ be a compact complex connected manifold.
If $X$ satisfies the $h$-principle, does this imply that 
$X$ is $\Bbb C$-connected ?
\end{question}

Combining Theorem \ref{hps} with the `Abelianity conjecture' of \cite {C01}, we obtain the following purely topological conjectural obstruction to the $h$-principle:

\begin{conjecture}\label{cab}  
Every projective manifold satisfying the $h$-principle 
should have an almost abelian fundamental group.
\end{conjecture}

Our proof of the implication ``$hp(X)\ \Longrightarrow$ {\em special}''
depends on `Joua\-nou\-lou's trick' which is not available for
non-algebraic manifolds.

Still we believe that the statement should also hold in the K\"ahler
case (for which specialness is defined as in definition 2.1 below):
\begin{conjecture}
Every compact K\"ahler manifold satisfying the $h$-principle should
be special.
\end{conjecture}

This implication might also hold for quasi-projective manifolds, 
provided their topology is sufficiently rich (non-contractibility being obviously a minimal requirement). Particular cases involving the quasi-Albanese map (dominance and connectedness) are established, using  \cite{NWY}. See theorems \ref{ndqa} and \ref{niqa} in \S \ref{QAm}.

The converse direction (``does specialness imply the $h$-principle?'')
is almost completely open. Based on classification and known results, 
the implication does hold for curves 
as well as surfaces which are either rational, 
or ruled over an elliptic curve, or blown-ups of either Abelian or bielliptic surfaces. 
The question remains open for all other special surfaces, and thus in particular for K3, even Kummer, surfaces. In higher dimensions even less is known,
e.g.~the case of $\P^3$ blown-up in a smooth curve of degree $3$ or more
is far from being understood.

Still, with a sufficient amount of optimism one might hope for
a positive answer to the question below.

\begin{question}
Let $X$ be a smooth (or at least normal) quasi-projective variety.
Assume that $X$ is either `special', or $\C$-connected.

Does it follow that $X$ satisfies the $h$-principle ? In this case, is it Oka (see \S\ref{EO})?
\end{question}

We present some examples showing that there is no positive answer
for arbitrary (ie: non-normal, non-K\" ahler, or non-algebraic varieties):
There are
examples of the following types which do {\em not} fulfill
the $h$-principle despite being $\C$-connected, or staisfying definition 2.1 (recall that we reserve the term `special' for the K\" ahler or quasi-K\" ahler case only):

\begin{enumerate}
\item
A non-normal projective curve which satisfies definition 2.1 and is $\C$-connected.
\item
A non-compact and non-algebraic complex manifold which is
$\C$-connected.
\item
A compact non-K\"ahler surface which satisfies definition 2.1.
\end{enumerate}

See \S \ref{EO}  for  more details.

\begin{remark} 
\begin{enumerate}
\item
Any contractible complex space trivially satisfies the $h$-principle.  
The notion ``$h$-principle'' is thus of interest only 
for non-contractible $X's$.
Since positive-dimensional compact manifolds are never contractible
this is not relevant for projective manifolds. However, 
there do exist examples of
contractible affine varieties of log general type (\cite{R},\cite{M})
indicating that
for non-compact varieties an equivalence ``hP $\iff$ {\em special}''
can hold only if the topology of the variety is sufficiently non-trivial.

\item
 Let $u:X'\to X$ be an unramified covering, with $X$ and $X'$ smooth 
and connected. Then $hP(X)$ implies $hP(X')$ (see Lemma \ref{et}), but the converse is not true.
To see this, consider a 
compact Brody-hyperbolic manifold $X$ which  
is an Eilenberg-MacLane $K(\pi,1)$-space,
but not contractible
(for example: a projective curve of genus $g\geq 2$
or a compact ball quotient). 
Then  its universal cover $\tilde X$ is contractible and therefore
satisfies the $h$-principle. On the other hand, being Brody-hyperbolic
and non-contractible, $X$ itself can not satisfy the $h$-principle.

\item For any given $X$ and $f$, possibly replacing the initial complex structure $J_0$ of $S$ by another one $J_1=J_1(f)$, homotopic to $J_0$, the existence of $F$ as in definition \ref{dip} above is always true (if $dim_{\Bbb C}(S)\geq 3$ at least. If $dim_{\Bbb C}(S)=2$, one must first compose with an orientation preserving homeomorphism of $S$). See \cite{F}, \S 9.10).
\end{enumerate}
\end{remark}

We thank Finnur L\'arusson for useful comments on an initial version of the present text.

\section{`Specialness'}\label{Sp}

\subsection{`Specialness' and the `core map'}

\

\

We refer to \cite{C01} for more details on this notion, to which the present section is an extremely sketchy introduction. Roughly speaking, special manifolds are higher-dimensional generalisations of rational and elliptic curves, thus `opposite' to manifolds of `general type' in the sense that they, and their finite \'etale covers, do not admit non-degenerate meromorphic maps to `orbifolds' of general type. Many qualitative properties of rational or elliptic curves extend or are expected to extend to `special' manifolds, although they are much more general (see remark \ref{rspec}.(7) below).

Let $X$ be a connected compact K\" ahler manifold.

\begin{definition} Let $p>0$, and $L\subset \Omega_X^p$ be a saturated rank $1$ coherent subsheaf. We define:

$$\kappa^{sat}(X,L):=\limsup_{m>0} 
\left\{\frac{log(h^0(X,\overline{mL}))}{log(m)}
\right\},$$ 
where $H^0(X,\overline{mL})\subset H^0(X,(\Omega_X^p)^{\otimes m})$ is the subspace of sections taking values in $L_x^{\otimes m}\subset (\Omega_X^p)_x^{\otimes m}$ at the generic point $x$ of $X$.

By a generalisation of Castelnuovo-De Franchis due to F. Bogomolov, $\kappa^{sat}(X,L)\leq p$, with equality if and only if $L=f^*(K_Y)$ at the generic point of $X$, for some meromorphic dominant map $f:X\dasharrow Y$, with $Y$ a compact $p$-dimensional manifold.

We say that $L$ is a `Bogomolov sheaf' on $X$ if $\kappa^{sat}(X,L)=p>0$, and that $X$ is `special' if it has no Bogomolov sheaf.
\end{definition}

\begin{remark}\label{rspec} 1. A `special' manifold is `very weakly special' (ie: has no dominant meromorphic map $f:X\dasharrow Y$ onto a positive-dimensional manifold $Y$ of `general type'), since $L:=f^*(K_Y)^{sat}$ would provide a Bogomolov sheaf on $X$. In particular, $X$ is not of general type (ie: $\kappa(X):=\kappa(X,K_X)<dim(X))$.

2. `Specialness' is a bimeromorphic property. If $X$ is special, so is any $Y$ `dominated' by $X$ (ie: such that a dominant rational map $f:X\dasharrow Y$ exists).

3. If $X$ is special, and if $f: X'\to X$ is unramified finite, then $X'$ is special, too. The proof (see \cite{C01}) is surprisingly difficult. It shows that `specialness' implies `weak specialness', defined as follows: $X$ is weaky special if any of its unramified covers is `very weakly special', as defined in (1) above.

4. The notion of `weak specialness' looks natural, and is easy to define. Unfortunately, it does not lead to any meaningfull structure result, such as the one given by the core map, stated below. On the other hand, it is also too weak to characterise the vanishing of the Kobayashi pseudometric (see (10) below).

5. Geometrically speaking, a manifold $X$ is `special' if and only if it has no dominant rational map onto an `orbifold pair' $(Y,\Delta)$ of general type. We do not define these concepts here. See \cite{C01} and \cite{C11} for details.

6. Compact k\" ahler manifolds which are either rationally connected, or with $\kappa=0$ are special (see \cite{C01}).

7. For any $n>0$ and any $\kappa\in \{-\infty, 0, 1,\dots, (n-1)\}$, there exists special manifolds with $dim(X)=n$ and $\kappa(X)=\kappa$. See, more precisely, \cite{C01}, \S 6.5.

8. For curves, `special' is equivalent to `very weakly special', and also to: non-hyperbolic. For surfaces, `special' is equivalent to `weakly special', and also to: $\kappa<2$, jointly with $\pi_1(X)$ almost abelian. Thus `special' surfaces are exactly the ones with either:

a. $\kappa=-\infty$ and $q\leq 1$, or:

b. $\kappa=0$, or:

c. $\kappa=1$, and $q(X')\leq 1$ for any finite \'etale cover $X'$ of $X$.

9. Another quite different characterisation of compact K\"ahler special surfaces $X$ is: $X$ is special if and only if it is $\Bbb C^2$-dominable. (with the possible exception of non-elliptic K3 surfaces, which are special, but not known to be $\Bbb C^2$-dominable). One direction is essentially due to \cite{BL}.

10. When $n:=dim(X)\geq 3$, there exists $X$ which are `special', but not `weakly special' (see \cite{BT}), and no simple characterisation of specialness 
depending only on $\kappa$ and $\pi_1$ does exist. 
Moreover, there are examples of weakly special varieties for which
the Kobayashi pseudometric does not vanish identically
(see \cite{CP}, \cite{CW}).
\end{remark}

The central results concerning `specialness' and having motivated its introduction
are the following two structure theorems 
(see \cite{C01} and \cite{C11} for definitions and details):

\begin{theorem} For any compact K\"ahler manifold $X$, there exists a unique almost holomorphic meromorphic map with connected fibres $c_X:X\dasharrow C(X)$ such that:

1. Its general fibre is special, and:

2. Its orbifold base $(C(X), \Delta_{c_X})$ is of general type (and a point exactly when $X$ is special).

The map $c_X$ is called the `core map' of $X$. It functorially `splits' $X$ into its parts of `opposite' geometries (special vs general type).
\end{theorem}

\begin{conjecture}
For any $X$ as above, $c_X=(J\circ r)^n$, where $n:=dim(X)$. Here $J$ (resp. $r$) are orbifold versions of the Moishezon fibration and of the `rational quotient' respectively. In particular, special manifolds are then towers of fibrations with general fibres having either $\kappa=0$, or $\kappa_+=-\infty$. 
\end{conjecture}

\begin{theorem}
The preceding conjecture holds, if the orbifold version of
Iitaka's $C_{n,m}$-conjecture is true.
\end{theorem}

\begin{remark} The above two theorems extend naturally to the full orbifold category. 
\end{remark}

The last (conditional) decomposition naturally leads (see \cite{C11}) to the following conjectures:

\begin{conjecture}\label{cj} 1. If $X$ is special, $\pi_1(X)$ is almost abelian.

2. $X$ is special if and only if its Kobayashi pseudometric vanishes identically.

3. $X$ is special if and only if any two of its points can be connected by an entire curve (ie: the image of a holomorphic map from $\Bbb C$ to $X$).

\end{conjecture}

\subsection{Orbifold Kobayashi-Ochiai and Factorisation through the core map}\label{ss-core}

The following orbifold version of Kobayashi-Ochiai extension theorem will be crucial in the proof of our main result.

\begin{theorem}\label{koo} (\cite{C01}, Theorem 8.2) Let $X$ be a compact K\"ahler manifold, $c_X: X\dasharrow C(X)$ be its core map\footnote{Or, more generally, any map $f:X\to Y$ of general type in the sense of \cite{C01}.}, $M\subset \overline{M}$ be a non-empty Zariski open subset of the connected complex manifold $\overline{M}$, and $\varphi:M\to X$ be a meromorphic map such that $g:=c_X\circ \varphi: M\to C(X)$ is non-degenerate (ie: submersive at some point of $M$). Then $g$ extends meromorphically to $\overline{M}$.
\end{theorem}

Applying this result to $M:=\Bbb C^n\subset \overline{M}:=\Bbb P^n$, we obtain that a non-degenerate meromorphic map $\varphi: \Bbb C^n\to X$ can exist only if $X$ is special. This is an indication in direction of the conjecture \ref{cj} (2) above.

\begin{theorem}\label{ftcm} Let $X, Z$ be complex projective manifolds and let $M$ be a smooth algebraic
variety admitting a surjective algebraic map $\tau: M\to Z$ with all fibers
affine spaces (isomorphic to $\Bbb C^k$). Let $G: M\dasharrow X$ be a meromorphic map, such that $g:=c_X\circ G:M\to C(X)$ is non-degenerate. Then $g$ also
factorises through $\tau$ and the core map $c_Z:Z\dasharrow C(Z)$ (ie: $g=\varphi\circ c_Z\circ \tau$, for some $\varphi: C(Z)\dasharrow C(X)$).
\end{theorem}

\begin{proof} $M$ can be compactified to a compact smooth projective
variety $\overline{M}$ by adding a hypersurface $D$ with normal crossings. 
By theorem \ref{koo} above, $g$ extends algebraically to $\bar g:\overline{M}\to C(X)$. Denote also by $\bar\tau: \overline{M}\to Z$ the extension of $\tau$ to $\overline{M}$. The orbifold base of the map $\bar g: \overline{M}\to C(X)$ is still $(C(X),\Delta_{c_X})$, and hence of general type in the sense of \cite{C01}, since it factorises through $X$ over $M$, and all the components of $D$ are mapped surjectively onto $C(X)$, since the fibres of $\tau$ are $\Bbb C^k$. 
The fact that the core map $c_{\overline{M}}$ dominates every general type fibration on $\overline{M}$ now yields
  a map $c_{\bar g}: C(\overline{M})\to C(X)$ such that $\bar g=c_{\bar g}\circ c_{\overline{M}}$. The map $\bar\tau$ induces also a map $c_{\bar\tau}: C(\overline{M})\to C(Z)$ such that $c_{\bar\tau}\circ c_{\overline{M}}=c_Z\circ \bar\tau$. Because the fibres of $\bar\tau$ are rationally connected, the map $c_{\bar\tau}$ is isomorphic, by \cite{C01}, Theorem 3.26. The composed map $\varphi:=c_{\bar g}\circ c_{\bar\tau}^{-1}:C(Z)\to C(X)$ provides the sought-after factorisation, since $\bar g=c_{\bar g}\circ c_{\overline{M}}=c_{\bar g}\circ c_{\bar\tau}^{-1}\circ c_Z\circ \bar\tau= \varphi\circ c_Z\circ \bar\tau$. \end{proof}

\begin{remark} The conclusion still holds if we replace $c_X$ by any fibration with general type orbifold base, and only assume that the fibres of $\bar\tau$ are rationally connected manifolds, and that all components of $D$ are mapped surjectively onto $Z$ by $\bar\tau$. This follows from \cite{GHS}, and \cite{C01}, theorem 3.26. 
\end{remark}

\section{Jouanoulou's trick}

\subsection{Jouanoulou's trick}

\begin{proposition}\label{jtrick}
Let $X$ be a projective manifold. Then there exists a smooth affine
complex variety $M$
and a surjective morphism  $\tau:M\to X$ such that
\begin{enumerate}
\item
$\tau:M\to X$ is a homotopy equivalence.
\item
Every fiber of $\tau$ is isomorphic to some $\C^n$.
In particular, every fiber has vanishing Kobayashi pseudodistance.
\item
$\tau$ is a locally holomorphically trivial fiber bundle.

\item

$\tau$ admits a real-analytic section.
\end{enumerate}
\end{proposition}

\begin{remark}
This is known as `Jouanoulou's trick' (see \cite J). This construction was introduced in Oka's theory in \cite{lar}, where the class $G$ of `good manifolds' is introduced, these being defined as the ones having a Stein affine bundle with fibre $\Bbb C^n$, for some $n$, observing that this class contains Stein manifolds, quasi-projective manifolds, and is stable by various usual geometric operations. 
\end{remark}

\begin{proof} We first treat the case of $X:=\Bbb P^N$, denoting with $\Bbb P^{N*}$ its dual projective space. Let $D\subset P:=\Bbb P^N\times \Bbb P^{N*}$ be the divisor consisting of pairs $(x,H)$ such that $x\in H$ (ie: the incidence graph of the universal family of hyperplanes of $\Bbb P^N$).  This divisor $D$ is ample, since intersecting positively the two family of lines contained in the fibres of both projections of $P$. Let $V$ be its complement in $P$. The projection $\tau_P$ on the first factor of $P$, restricted to $V$, satisfies the requirements for $X:=\Bbb P^N$. A real-analytic section is obained by choosing a hermitian metric on $\Bbb C^{n+1}$, and sending a complex line to its orthogonal hyperplane.

In the general case, embed first $X$ in some $\P_N$. 
Let then $M=\tau_P^{-1}(X)$ and let $\tau$ denote the restriction of
$\tau_P$ to $M$. 
Now $M$ is a closed algebraic subset of $V$ and therefore
likewise affine. Everything then restricts from $\Bbb P^N$ to $X$.
\end{proof}

Remark that, when $X=\Bbb P^1$, we recover the two-dimensional
affine quadric as $M$ 
(and indeed, $\Bbb P^1$ is diffeomorphic to $S^2$).

If $X$ is a projective curve, we may obtain a bundle $M\to X$
with the desired properties also in a different way:

Let $Q_2=\Bbb P^1\times \Bbb P^1-D$, where $D$ is the diagonal. 
Taking the first projection, we get an affine bundle $Q_2\to \Bbb P^1$ with fibre $\Bbb C$ over $\Bbb P^1$, which is an affine variety. 
Now we choose a finite morphism $f$ from $X$ to $\P_1$ and
define $M\to X$ via base change.

\begin{question}
Given a complex manifold $Z$, does there exists a Stein manifold $S$ and
a holomorphic map $f:S\to Z$ whose fibers are isomorphic to $\C^n$ ? Is this true at least when $Z$ is compact K\" ahler?
\end{question}

\section{Opposite complex structures and associated cohomological integrals}

\subsection{Inverse images of forms under meromorphic maps}

\begin{lemma}\label{pull-back-mero}
Let $f:X\to Y$ be a dominant meromorphic map between compact complex manifolds, $\dim X=n$, with $I(f)\subset X$ being the indeterminacy set.
For every $c\in H^{k,k}(Y)$ there exists a unique cohomology class $c'\in H^{k,k}(X)$
such that:
\[
[\alpha].c'=\int_{X\setminus I(f)}\alpha\wedge f^*\beta
\]
for every closed smooth $(n-k,n-k)$-form $\alpha$ on $X$
and every closed smooth $(k,k)$-form $\beta$ with $[\beta]=c$.

We define the inverse image of the De Rham cohomology class $[c]$
with respect to the meromorphic map $f$ by: $f^*([c]):= c'$.
\end{lemma}
\begin{proof}
Let $\tau:X'\to X$ be a blow up such that $f$ lifts to a holomorphic map $F:X'\to Y$.
Using Poincar\'e duality, $F^*\beta$ may be identified with a linear form
on $H^{n-k,n-k}(X')$. Restricting this linear form to
$\tau^*H^{n-k,n-k}(X)$ and again using Poincar\' e duality there is a unique
cohomology class $c'$ such that:
\[
[\alpha].c'=\int_{X'}\alpha\wedge F^*\beta.
\]
Furthermore
\[
\int_{X'}\alpha\wedge F^*\beta=\int_{X\setminus I(f)}\alpha\wedge f^*\beta
\]
since $\alpha\wedge f^*\beta$ is a top degree form and both the exceptional divisor
of the blow up and the indeterminacy set $I(f)$ of the meromorphic $f$ are sets
of measure zero.
\end{proof}

From the characterization of this inverse image, it is clear 
that is is compatible
with composition of dominant meromorphic maps. It is also clear 
that it specializes
to the usual pull-back if the meromorphic map under 
discussion happens to be holomorphic.

(Caveat: This inverse image gives linear maps between the cohomology groups, but
(as can be seen by easy examples) it does not define a ring homomorphism 
between the
cohomology rings.)

\subsection{Opposite complex structures}

Given a complex manifold $X$ we define the {\em opposite}, or {\em conjugate} complex
manifold (also called {\em opposite complex structure} on $M$)
as follows: If $X_0$ is the underlying real manifold and $J$ is the
almost complex structure tensor of $X$, we define as the opposite complex
manifold $X_0$ equipped with $-J$ as complex structure tensor.
Recall that an almost complex structure is integrable if and only
if the Nijenhuis-tensor vanishes. This implies immediately that
$(X_0,-J)$ is also a complex manifold (i.e.~$-J$ is an {\em integrable}
almost complex structure). One can also argue directly without Newlander-Nirenberg's theorem. 

Now consider the complex projective space $\P_n(\C)$.
The map
\[
[z_0:\ldots:z_n]\mapsto [\bar z_0:\ldots:\bar z_n]
\]
defines a biholomorphic map between $\P_n(\C)$ and its opposite.
As a consequence, we deduce: If a complex manifold $X$ is projective,
so is its opposite $\overline{X}$.

Now assume $X$ admits a K\"ahler form $\omega$.
Then the opposite complex manifold $\bar X$ is again
a K\"ahler manifold. Indeed, since $\omega(v,w)=g(Jv,w)$ defines
the K\"ahler form on a complex manifold admitting a Riemannian metric
$g$ for which $J$ is an isometry, we see that $\bar X$ admits a K\"ahler
metric with $-\omega$ as K\"ahler form. The same property applies if $g$ is, more generally, a hermitian metric on $X$, and $\omega$ its associated `K\" ahler' form, defined from $J, g$ by the formula above.

{\em Orientation}. On a K\"ahler manifold $X$ with K\"ahler form
$\omega$ the orientation is defined by imposing that $\omega^n$
is positively oriented where $n=\dim_{\C} X$.
This implies: If $X$ is a K\"ahler manifold and $\bar X$ is its opposite,
the identity map of the underlying real manifold defines an
orientation preserving diffeomorphism if $n=\dim_{\C}(X)$ is even
and an orientation reversing one if $n$ is odd.

\subsection{Inverse image of forms and opposite complex structures}

\begin{lemma}\label{int} Let $X$ be an $n$-dimensional compact complex manifold, $\overline{X}$ its conjugate, and $\zeta: \overline{X}\to X$ a smooth map homotopic to the identity map $id_X$ of $X$. Let $c:X\dasharrow Y$ be a meromorphic map to a complex manifold $Y$. Let $c\circ \zeta=:\varphi: \overline{X}\to Y$. Let $\alpha$ be a $d$-closed smooth differential form of degree $2d$ on $Y$, and $\omega_X$ a smooth closed $(1,1)$-form on $X$. Then:

$I'=:\int_{\overline X} \zeta^*(\omega_X^{n-d}\wedge c^*(\alpha))=(-1)^{d}.\int_X \omega_X^{n-d}\wedge c^*(\alpha):=(-1)^d.I$
\end{lemma}

\begin{proof} From the above remarks on the orientations of $X$ and $\overline{X}$, and the fact that $id_X^*(\omega_X)=-\omega_{\overline{X}}$, we get:
$I=(-1)^n\int_{\bar X} 
\omega_X^{n-d}
\wedge c^*\alpha$.

Since $\zeta$ is homotopic to
$id_X$, and $c\circ \zeta=\varphi$, we get:

$I=(-1)^n\int_{\bar X} 
\zeta^*(\omega_X^{n-d}
\wedge c^*(\alpha))$

$= (-1)^n\int_{\bar X} 
\zeta^*(\omega_X^{n-d})
\wedge\varphi^*(\alpha)$
$=(-1)^n\int_{\bar X} 
(-1)^{n-d}\omega_{\bar X}^{n-d}
\wedge \varphi^*(\alpha) $

$=(-1)^d\int_{\bar X} 
\omega_{\bar X}^{n-d}
\wedge \varphi^*(\alpha)=(-1)^d. I'$
\end{proof}

\begin{corollary} \label{cint} In the situation of the preceding Lemma \ref{int} , assume that $X$ is compact K\" ahler, that $dim(Y)>0$, and that $c: X\dasharrow Y$ is non-degenerate (ie: dominant). Then $\varphi:=c\circ \zeta: \overline{X}\dasharrow Y$ is not meromorphic.
\end{corollary} 

\begin{proof} Assume $\varphi$ is meromorphic. After suitable modifications, we may assume that $Y$ is K\" ahler. Let $\alpha:=\omega_Y$ be a K\" ahler form on $Y$. Choose $d=1$ in Lemma \ref{int}. Then $I:=\int_{ X} 
\omega_X^{n-1}
\wedge c^*(\omega_Y)>0$. On the other hand, $I':=\int_{ \overline{X} }
\omega_{\overline{X} }^{n-1}
\wedge \varphi^*(\omega_Y)>0$. From Lemma \ref{int} we deduce: $I'=-I$ and a contradiction.
\end{proof}

\section{$h$-principle and Brody-hyperbolicity}
\subsection{$h$-principle and weak $\Bbb C$-connectedness}

\begin{proposition}\label{stein-sphere}
For any $n>0$, the $n$-dimensional sphere $S^n$ is homotopic to the (complex) $n$-dimensional affine quadric $Q_n(\Bbb C)$ defined by the equation \[
Q_n=\left\{z=(z_0,\ldots,z_n)\in \Bbb C^{n+1}:\sum_k z_k^2=1\right\},
\] 

Any two points of $Q_n$ are connected by an algebraic $\Bbb C^*$, and so its Kobayashi
pseudometric vanishes identically.
\end{proposition}

\begin{proof} Let $q$ be the standard non-degenerate quadratic form in $\Bbb R^{n+1}$. The set $Q_n(\Bbb R)$ of real points of $Q_n$ obviously coincides with $S^n$. An explicit real analytic isomorphism $\rho:Q_n\to N_n$ with the real normal (ie: orthogonal) bundle $N_n:=\{(x,y)\in S^n\times \Bbb R^{n+1}: q(x,y)=0\}$ of $S^n$ in $\Bbb R^{n+1}$, is given by: $\rho(z=x+i.y):=(\lambda(z).x,\lambda.y)$, where $\lambda(z)^{-1}:=\sqrt{1+q(y,y)}$. The map $\rho$ is in particular a homotopy equivalence.

The last assertion is obvious, since any complex affine plane in $\Bbb C^{n+1}$ intersects $Q_n$ either in a conic with one or two points deleted, or in a two-dimensional complex affine space.
\end{proof}

\begin{question}
Let $Z$ be a connected differentiable manifold or a finite-dimensional 
$CW$-complex.
Does there exist topological obstructions to the existence of a Stein manifold $S$ homotopic to $Z$ with vanishing
Kobayashi pseudodistance?

In particular, does there exist a Stein manifold with vanishing Kobayashi pseudodistance (eg. $\C$-connected) and homotopic to a smooth connected projective curve of genus $g\geq 2$?
\end{question}

The main difficulty here is the condition on the Kobayashi pseudodistance.
In fact, it is not too hard to to give an always positive answer if one drops
the condition on the Kobayashi pseudodistance:

\begin{proposition}
Let $Z$ be a connected differentiable manifold or a finite-dimensional $CW$-complex (as always with countable base of topology).
Then 
there exists a Stein manifold $M$ homotopic to $Z$.
\end{proposition}
\begin{proof} This is a known consequence of the classical characterisation of Stein spaces by H. Grauert (see \cite{F}, corollary 3.5.3, and the references there, for example). We give here a short proof, using a deep theorem of Eliashberg.

If $Z$ is a $CW$-complex, we embedd into some $\R^n$. Then $Z$ is
homotopic to some open neighborhood of $Z$ in $\R^n$. Since open subsets
of $\R^n$ are manifolds, it thus suffices to deal with the case
where $Z$ is a differentiable manifold. By taking a direct product
with some $\R^k$, we may furthermore assume that $\dim_{\R}(Z)>2$.
Let $M=T^*Z\stackrel{\tau}\mapsto Z$ denote the cotangent bundle.
Then $M$ carries a symplectic structure in a natural way and therefore
admits an almost complex structure.
Fixing a metric $h$ on $M=T^*Z$ and choosing an exhaustive Morse function
$\rho$ on $Z$, we can use $p(v)=\rho(\tau(v))+h(v)$
as an exhaustive Morse function on $M$.
By construction the critical points of $p$ are all in the zero-section
of the cotangent bundle of $Z$ and coincide with the critical points
of $\rho$. Therefore there is no critical point of index 
larger than $\dim(Z)=\frac12\dim M$.
By a result of Eliashberg (\cite{E}) it follows from the existence of such a Morse
function and the existence of an almost complex structure 
that $M$ can be endowed with
the structure of Stein complex manifold.
This completes the proof since $M$ is obviously homotopy equivalent
to $Z$.
\end{proof}

\begin{theorem}\label{thpcc}
Let $X$ be a complex space which fulfills the $h$-principle. Then $X$ is `weakly $\Bbb C$-connected'.
\end{theorem}

\begin{proof}
Assume not. Since $hP(X)$ is preserved by passing to unramified coverings
(see lemma~\ref{et}), we may assume that $X'=X$ in definition \ref{dcc}(3). Then  there exists a holomorphic map $g:X\to Y$, with $Y$ Brody-hyperbolic, and such that there exists a non-zero induced homotopy map $\pi_k(g):\pi_k(X)\to \pi_k(Y), k>0$.
Let $f:S^k\to X$ be a continuous map defining a non-trivial
element of $g\circ f:S^k\to \pi_k(Y)$, where $S^k:=$ the $k$-dimensional sphere.
Let $Q_k$ be the $k$-dimensional affine quadric,
and a continuous map $\varphi:Q_k\to S^k$ which is a homotopy equivalence
(its existence is due to proposition~\ref{stein-sphere}).
Then $f\circ\varphi:Q_k\to Y$ is a continuous map which is not homotopic
to a constant map. But due to the Brody-hyperbolicity of $Y$, 
every holomorphic map from $Q_k$ to $Y$ must be constant, contradicting our initial assumption.
\end{proof}

Applying the preceding result to $Y:=X$, we get:

\begin{corollary}\label{bhnhp}
Let $X$ be a Brody-hyperbolic complex manifold.

Then $X$ fulfills the $h$-principle if and only if it is contractible.
\end{corollary}

\begin{corollary}
Let $X$ be a positive-dimensional compact complex Brody-hyperbolic manifold.
Then $X$ does not fulfill the $h$-principle.
\end{corollary}

\begin{proof}
Positive-dimensional compact manifolds are not contractible
\end{proof}

In particular, compact Riemann surfaces of genus $g\ge 2$ do not fulfill
the $h$-principle.

\begin{remark}
There exist holomorphic maps $f:X\to Y$
with $X$ and $Y$ both smooth and projective
which are not homotopic to a constant map, although
$\pi_k(f)=0$  for all $k>0$. 
For example, take a compact Riemann surface $X$ of genus $g\ge 2$ and let
$f$ be any non-constant map to $\Bbb P^1$ (example suggested by F. Bogomolov). 

Therefore it is not clear whether the property ``weakly $\C$-connected''
implies that every holomorphic map to a Brody-hyperbolic complex
space must be homotopic to a constant map.

The following theorem \ref{tphpwcc} solves this issue in the projective case, assuming the $h$-principle.
\end{remark}

\subsection{Projective Brody-hyperbolic quotients}

\begin{theorem}\label{tphpwcc} Let $X$ be an irreducible projective complex space satisfying the $h$-principle. 
Let $f:X\dasharrow Y$ be a meromorphic map
to a Brody hyperbolic K\"ahler manifold $Y$.

Assume that $f$ is holomorphic or that $X$ is smooth.

Then $f$ is constant.
\end{theorem}

\begin{proof} 
For every meromorphic map $f:X\dasharrow Y$ there exists a 
proper modification
$\hat X\to X$ such that $f$ can be lifted to a holomorphic map defined
on $\hat X$. If $X$ is smooth, this modification can be obtained by
blowing-up smooth centers, implying that the fibers of $\hat X\to X$
are rational. Since $Y$ is Brody-hyperbolic, holomorphic maps
from rational varieties to $Y$ are constant.

Hence $X$ being smooth implies that $f$ is already holomorphic.

Thus in any case, we may assume that $f$ is holomorphic.

Because $X$ is projective, we may find a compact complex curve $C$ on $X$
such that $f|_C$ is non-constant.

Let $\bar C$ be $C$ equipped with its 
conjugate (ie: opposite) complex structure, and 
$j:\bar C\to C$ be the set-theoretic identity map. 
Let $\tau:E\to \bar C$ be an holomorphic affine $\C$-bundle
as given by proposition~\ref{jtrick}.

Since $X$ is assumed to fulfill the
$h$-principle, the continuous map 
$j\circ \tau: E\to X$ 
is homotopic to a holomorphic map $h:E\to X$. 
Because $Y$ is Brody hyperbolic, the map $f\circ h:E\to Y$ is constant 
along the fibres of $\tau$.
Hence $f\circ h$ is equal to $\varphi\circ \tau$ for a holomorphic map 
$\varphi: \bar C\to Y$.
Observe that $\varphi,f\circ j : \bar C\to Y$ are homotopic too each
other, but the first map is holomorphic while the latter is
antiholomorphic. This is a contradiction, because now
\[
0 < \int_{\bar C}\varphi^*\omega = 
\int_{\bar C}(f\circ j)^*\omega < 0
\]
for any K\"ahler form $\omega$ on $Y$.
\end{proof}

\section{$h$-principle implies specialness for projective manifolds}
\begin{theorem}\label{hps}
Let $X$ be a complex projective manifold.

If $X$ fulfills the $h$-principle, then $X$ is special
in the sense of \cite{C01}.
\end{theorem}
\begin{proof}
Let $\bar X$ denote the underlying real manifold equipped
with the opposite complex structure and let $id_X:
\bar X\to X$ denote the antiholomorphic diffeomorphism
induced by the identity map of this underlying real manifold.

Recall that $\bar X$ is projective, too.
Hence we can find a Stein manifold $M$ together with a holomorphic
fiber bundle $\tau:M\to\bar X$ with some $\C^k$ as fiber
(proposition~\ref{jtrick}).

Let $\sigma:\bar X\to M$ denote a smooth (real-analytic, for example) section
 (whose existence is guaranteed by proposition~\ref{jtrick}).

Since we assumed that $X$ fulfills the $h$-principle, there must
exist a holomorphic map $h:M\to X$ homotopic to $id_X\circ\tau$.

Define $\zeta:=h\circ \sigma: \bar X\to X$. Thus: $\zeta$ is homotopic to $id_X$.

Let $c:X\dasharrow C$ be the core map of $X$. We assume that $X$ is not special, i.e., that $d:=\dim(C)>0$. Let also: $n=\dim X$.

We claim that $c\circ \zeta:\bar X\dasharrow C$ is non-degenerate, and thus, that so is $g:=c\circ h:M\to C(X)$.

Let indeed, $\omega_C$ (resp. $\omega_X)$ be a K\" ahler form on $C$ (resp. on $X$), and let $d:= dim(C)$. Then $I:=\int_X\omega_X^{n-d}\wedge c^*(\omega_C^d)>0$. By lemma \ref{int}, we have: $I':=\int_{\bar X} \zeta^*(\omega_X^{n-d}\wedge c^*(\omega_C^d))=(-1)^d.I\neq 0$. This implies that $(c\circ\zeta)^*(\omega_C^d)\neq 0$, and so that $c\circ \zeta$ is not of measure zero. By Sard's theorem, this implies that $c\circ \zeta$ is non-degenerate, and so is thus $c\circ h$.

We consider the meromorphic map $c\circ h:=g:M\to C$.
By theorem \ref{ftcm},  it follows that we obtain an induced meromorphic map
$\varphi:\bar X\to C$ such that $\varphi\circ\tau= g$, and thus such that: $\varphi=\varphi\circ \tau\circ \sigma=c\circ h\circ \sigma=c\circ \zeta$.

We consider now the integral:
$J=\int_X \omega_X^{n-1}\wedge c_X^*(\omega_C)$. Thus $J>0$.

From corollary \ref{cint} we get a contradiction.

Hence $X$ cannot fulfill the $h$-principle, unless $\dim(C)=0$,
i.e. unless $X$ is special.
\end{proof}

A consequence of theorem \ref{hps} and conjecture \ref{cj} is the following homotopy restriction for the $h$-principle to hold:

\begin{conjecture}\label{ab} If $X$ is complex projective manifold satisfying the $h$-principle, then $\pi_1(X)$ is almost abelian.
\end{conjecture}

Notice that this conjecture is true if $\pi_1(X)$ has a faithfull linear representation in some $Gl(N,\Bbb C)$, or is solvable, by \cite{C11}, and 
\cite{C10} respectively. 

The above result on projective manifolds rises the following questions.

\begin{question} 
\begin{enumerate}
\item
 Are compact K\"ahler manifolds satisfying the $h$-principle special? This is true, at least, for compact K\"ahler surfaces (see proposition \ref{hpws} and its corollary below).
\item
Let $X$ be a quasi-projective manifold satisfying the $h$-principle.
Assume that $X$ is not homopy-equivalent to any proper subvariety $Z\subset X$.
Does it follow that $X$ is special?
\end{enumerate}
\end{question}

We have some partial results towards answering these questions.

\begin{theorem}\label{hpws} Let $X$ be a compact K\"ahler manifold satisfying the $h$-principle. Then the Albanese map of $X$ is surjective.
\end{theorem} 

\begin{proof} The proof of theorem \ref{ndqa} applies
\end{proof}

\begin{corollary} Let $X$ be a compact K\"ahler surface satisfying the $h$-principle. Then $X$ is special. 
\end{corollary} 

\begin{proof} Assume not. Then $X$ is in particular not weakly special. Since $X$ is not of general type, by theorem \ref{hps}, there exists a finite \'etale cover $\pi:X'\to X$ and a surjective holomorphic map $f:X'\to C$ onto a curve $C$ of general type. Because $X'$ also satisfies the $h$-principle, by Lemma \ref{et} below, this contradicts theorem \ref{hpws}
\end{proof}

\begin{lemma}\label{et}
Let $\pi:X'\to X$ be an unramified covering between complex spaces. If $X$ fulfills the $h$-principle,
so does $X'$.
\end{lemma}
\begin{proof}
Let $f:S\to X'$ be a continuous map from a Stein space $S$.
By assumption, there is a holomorphic map $g:S\to X$ homotopic to 
$\pi\circ f$. The homotopy lifting property for coverings implies
that $g$ can be lifted to a holomorphic map $G:S\to X'$ which is
homotopic to $f$.
\end{proof}

\section{Necessary conditions on the Quasi-Albanese map}\label{QAm}

We give two necessary conditions, bearing on its quasi-Albanese map, in order that a quasi-projective manifold $X$ satisfies the $h$-principle. These conditions are necessary for $X$ to be special.

\begin{theorem}\label{ndqa}
Let $X$ be a complex quasi projective manifold for which the Quasi-Albanese map
is not dominant.

Then $X$ does not satisfy the $h$-principle.
\end{theorem}

\begin{proof}
Let $A$ be the Quasi Albanese variety of $X$ and let $Z$ denote the closure of the image of
$X$ under the Quasi Albanese map $a:X\to A$. We may assume $e_A\in Z$.
By the theorem of Kawamata (\cite{K}), there are finitely many subtori $T_i\subset A$
and $T_i$-orbits $S_i\subset A$ such that $S_i\subset Z$ and such that
every translated subtorus of $A$ which is contained in $Z$ must already
be contained in one of the $S_i$. 
Due to lemma~\ref{lemx} below, there is an element
$\gamma_0\in\pi_1(A)$ which is not contained in any of the $\pi_1(S_i)$.
By the functoriality properties of the Albanese map the group
homomorphism $\pi_1(X)\to\pi_1(A)$ is surjective. Thus we can lift
$\gamma_0$ to an element $\gamma\in\pi_1(X)$.
Let us now assume that the $h$-principle holds.
In this case there must exist a holomorphic map $f$ from $\C^*$ to $X$
inducing $\gamma$. By composition we obtain a holomorphic map
\[
F=a\circ f \circ\exp:\C\to Z\subset A
\]
Now Noguchis logarithmic version of the theorem of Bloch-Ochiai implies that the analytic
Zariski closure of $F(\C)$ in $Z$ is a translated sub semitorus of $A$.
Therefore $F(\C)$ must be contained in one of the $S_i$.
But this implies
\[
(a\circ f)_*\left(\pi_1(\C^*)\right)\subset\pi_1(S_i)
\]
which contradicts our choice of $\gamma$.
\end{proof}

\begin{lemma}\label{lemx}
Let $\Gamma_1,\ldots,\Gamma_k$ be a family of subgroups of $G=\Z^n$
with $rank_{\Z}\Gamma_i<n$.

Then $\cup_i\Gamma_i\ne G$.
\end{lemma}

\begin{proof}
For a subgroup $H\subset G\subset\R^n$ let $N(H,r)$ denote the number of elements
$x\in H$ with $||x||\le r$. Then $N(H,r)=O(r^d)$ if $d$ is the rank of the $\Z$-module
$H$. Now $N(\Gamma_i,r)=O(r^{n-1})$, but $N(G,r)=O(r^n)$. This implies the
statement.
\end{proof}

We find again:

\begin{corollary}
Let $X$ be an algebraic variety which admits a surjective morphism
onto an algebraic curve $C$. If $C$ is hyperbolic, then $X$ does not
fulfill the $h$-principle.
\end{corollary}

\begin{proof}
Let $A$ resp.~$J$ denote the quasi-Albanese variety of $X$ resp.~$C$.
By functoriality of the quasi Albanese we have a commutative diagram

Since $\dim(J)>\dim(C)$ due to hyperbolicity of $C$, the quasi-Albanese
map $X\to A$ can not be dominant.
\end{proof}

By similar reasoning, using \cite{NWY}:

\begin{proposition}\label{niqa}
Let $X$ be a quasi projective manifold which admits a finite map onto an semi abelian
variety.

Then $X$ fulfills the $h$-principle only if $X$ is a semi-abelian variety.
\end{proposition}

\section{(Counter-)examples}\label{CE}

We now present some examples showing that the desired
implications ``special $\implies$ $h$-principle''
and ``$\C$-connected $\implies$ special'' certainly
do not hold without imposing some normality and 
algebraicity/K\"ahler
condition on the manifold in question.

\begin{example}
There is a non-normal projective curve $X$ which is rational and
$\C$-connected, but does not fulfill the $h$-principle.

We start with $\hat X=\P_1$ and define $X$ by identifying $0$ and 
$\infty$ in $\hat X=\C\cup\{\infty\}$. Via the map
$[x_0:x_1]\mapsto[ x_0^3+x_1^3:x_0^2x_1:x_0x_1^2]$ the quotient
space $X$ can be realized as
\[
X\simeq\{[z_0:z_1:z_2]: z_0z_1z_2=z_1^3 + z_2^3\}.
\]
Let $\tilde X$ denote the universal covering of $X$. Then $\tilde X$
consists of countably infinitely many $2$-spheres glued together.
By Hurewitz $\pi_2(\tilde X)\simeq H_2(\tilde X,\Z)\simeq\Z^\infty$.
The long homotopy sequence associated to the covering map
implies $\pi_2(X)\simeq\Z^\infty$.
As a consequence the group homomorphism
\[
\Z\simeq \pi_2(\hat X)\ \longrightarrow\ \pi_2(X)\simeq\Z^\infty
\]
induced
by the natural projection $\pi:\hat X\to X$ is not surjective.
Now let $Q$ denote the two-dimensional affine quadric. Note that
$Q$ is a Stein manifold which is homotopic to the $2$-sphere.
Because $\pi_2(\hat X) \to \pi_2(X)$  is not surjective, there
exists a continuous map $f:Q\to X$ which can not be lifted to a
continuous map from $Q$ to $\hat X$. On the other hand, every holomorphic
map from the complex manifold $Q$ to $X$ can be lifted to $\hat X$,
because $\hat X$ is the normalization of $X$.
Therefore there exists a continuous map from $Q$ to $X$ which is not
homotopic to any holomorphic map. Thus $X$ does not fulfill
the $h$-principle.
\end{example}

\begin{example}
There are non-K\"ahler compact surfaces, namely Inoue surfaces,
which do not fulfill the $h$-principle, although they are special.
 
These Inoue surfaces are compact
complex surface of algebraic dimension zero with $\Delta\times\C$
as universal covering and which are foliated by complex lines.
They are special (meaning that they satisfy definition 2.1, but the term `special' is reserved strictly speaking to the compact K\" ahler case), 
because due to algebraic dimension zero there
are no Bogomolov sheaves. On the other hand, every holomorphic map
from $\C^*$ to such a surface has its image contained in one of those
leaves. This implies that there are many group homomorphisms from
$\Z$ to the fundamental group of the surface which are not
induced by holomorphic maps from $\C^*$. For this reason Inoue surfaces
do not fulfill the $h$-principle.
\end{example}

\begin{example}
There is a non-compact complex manifold which is $\C$-connected,
but does not satisfy the $h$-principle.

Due to Rosay and Rudin (\cite{RR})
there exists a discrete subset $S\subset\C^2$
such that $F(\C^2)\cap S\ne\{\}$ for any non-degenerate holomorphic map
$F:\C^2\to\C^2$. (Here $F$ is called non-degenerate iff there is a point
$p$ with $rank(DF)_p=2$.)
Let $X=\C^2\setminus S$.
Due to the discreteness of $S$ it is easy to show that
$X$ is $\C$-connected. Now let $G=SL_2(\C)$. Then $G$ is a Stein
manifold which is homotopic to $S^3$.
Let $p\in SL_2(\C)$ and $v,w\in T_pG$. Using the
exponential map there is a holomorphic map from $\C^2$ to $G$ for which
$v$ and $w$ are in the image. From this it follows easily
that for every holomorphic
map $F:G\to X$ and every $p\in G$ we have $rank(DF)_p\le 1$.
Hence $F^*\omega\equiv 0$ for every $3$-form $\omega$ on $X$ and
every holomorphic map $F:G\to X$.
This implies that for every holomorphic map $F:G\to X$ the induced
map $F^*:H^3(X,\R)\to H^3(G,\R)$ is trivial. 
On the other hand there are continuous maps $f:S^3\to X$
for which $f^*:H^3(X,\C)\to H^3(S^3,\C)$ is non-zero:
Choose $p\in S$. Since $S$ is countable, there is a number $r>0$
such that $||p-q||\ne r\ \forall q\in S$.
Then $f:v\mapsto p+rv$ defines a continuous from $S^3=\{v\in\C^2:||v||=1\}$
to $X$ which induces a non-zero homomorphism $f^*:H^3(X,\C)\to H^3(S^3,\C)$.

As a consequence, $X$ does not fulfill the $h$-principle.
\end{example}

\section{``special'' $\implies$ $h$-principle ?}\label{EO} 

We consider the question: if $X$ is projective, smooth and special, does it satisfy the $h$-principle? The question is very much open, even in dimension $2$.

For projective curves, we have the equivalence: $h$-principle satisfied if and only if special.

The projective surfaces known to satisfy the $h$-principle are the following ones: the rational surfaces, the minimal surfaces ruled over an elliptic curve, the blown-up Abelian surfaces and their \'etale undercovers, termed `bielliptic'.

This means that the special projective surfaces not known to satisfy the $h$-principle are, on the one hand, the blown-up K3 and Enriques surfaces, and on the other hand the blown-ups of surfaces with $\kappa=1$, which are elliptic fibrations over, either:

\begin{enumerate}
\item
 an elliptic base, and without multiple fibre, or:
\item 
 a rational base, and with at most $4$ multiple fibres, the sum of the inverses of the multiplicities being at least $2$ (resp. $1$) if there are $4$ (resp. $3$) multiple fibres.
\end{enumerate}
 
 In higher dimension (even $3$), essentially nothing is known. In particular, the cases of Fano, rationally connected, and even rational manifolds (for example: $\Bbb P^3$ blown-up along a smooth curve of degree $3$ or more) is open.
 
 For $n$-dimensional Fano or rationally connected manifolds, $n\geq 3$, even the existence of a non-degenerate meromorphic map from $\Bbb C^n$  to $X$ is open. This inexistence would contradict the Oka property (see definition below). In case such a map exists, nothing is known about the unirationality of $X$ (see \cite{U}, and \cite{C01}, for example).

Let us first remark that the $h$-principle satisfaction is not known to be preserved by many standard geometric operations preserving specialness. In particular, this concerns:

\begin{enumerate}
\item
Smooth blow-ups and blow-down.
\item
For (finite) \'etale coverings only one direction is known (cf.~X).
\end{enumerate}

Except for trivial cases it is very hard to verify the $h$-principle
directly.
The most important method for verifying the $h$-principle is
Gromov's theorem that
the $h$-principle is satisfied by `elliptic manifolds'.
In the terminology of M. Gromov ``ellipticity'' 
 means the existence of a holomorphic vector bundle $p:E\to X$ with zero section $z:X\to E$, and a holomorphic map $s: E\to X$ such that $s\circ z:X\to X$ is the identity map, and the derivative $ds:E\to TX$ is surjective along $z(E)$, where $E\subset TE$ is the kernel of the derivative $dp: TE\to TX$ along $z(X)\subset E$.

Homogeneous complex manifolds (e.g.~$\P_n$, Grassmannians, tori) are
examples of elliptic manifolds. Complements $\C^n\setminus A$ of algebraic
subvarieties $A$ of codimension at least two are also known to be
elliptic.

For a complex manifold $X$ being elliptic also implies  
that $X$ is `Oka', i.e.: every holomorphic map $h:K\to X$ on a compact convex subset $K$ of $\Bbb C^n$ can be uniformly approximated to any precision by holomorphic maps $H:\Bbb C^ n\to X$. Forstneric's theorems 
(\cite{F}) show that Oka manifolds satisfy stronger approximation properties. 
All known examples of Oka manifolds are subelliptic, a slight weakening of ellipticity. We refer to \cite{G}, \cite{F}, and \cite{FL} for more details and generalisations of these statements.  See also \cite{L} for an interpretation of the Oka property in terms of `Model structures'.

We have thus the following sequence of implications (the first two being always valid, the last for projective manifolds):

\[
\text{elliptic} \Rightarrow \text{Oka} \Rightarrow \text{$h$-principle} \Rightarrow 
\text{special}
\]

Although the notions `Oka' and `$h$-principle satisfied' differ in general
(for example the unit disc is evidently not Oka, but satisfies
the $h$-principle, because it is contractible), one may ask:

\begin{question} Is any projective manifold satisfying the $h$-principle Oka?
\end{question}

 \end{document}